\documentclass{louloupart1}
\usepackage{newtxtext,newtxmath}
\usepackage{pinlabel}
\usepackage{graphicx}
\usepackage[all]{xy}
\usepackage{xcolor}

\title[The braid group $\BB_n$ is not a quotient of a quasi-Coxeter interval group of type $D_n$]{The braid group $\BB_n$ is not a quotient of a quasi-Coxeter interval group of type $D_n$}
\author[B. Baumeister]{Barbara Baumeister}
\givenname{Barbara}
\surname{Baumeister}
\address{Barbara Baumeister, Fakultät für Mathematik Universität Bielefeld, Postfach 10 01 31, 33501 Bielefeld, Germany.}
\email{b.baumeister@math.uni-bielefeld.de}
\author[L. Paris]{Luis Paris}
\givenname{Luis}
\surname{Paris}
\address{Luis Paris, Universit\'e Bourgogne Europe, CNRS, IMB, UMR 5584, 21000 Dijon, France.}
\email{lparis@u-bourgogne.fr}
\author[S. Rees]{Sarah Rees}
\givenname{Sarah}
\surname{Rees}
\address{Sarah Rees, School of Mathematics, Statistics and Physics, University of Newcastle, Newcastle NE1 7RU, UK.}
\email{sarah.rees@newcastle.ac.uk}

\newtheorem{thm}{Theorem}[section]
\newtheorem{lem}[thm]{Lemma}
\newtheorem{prop}[thm]{Proposition}
\newtheorem{corl}[thm]{Corollary}
\newtheorem{conj}[thm]{Conjecture}

\newtheorem{claim}{Claim}

\theoremstyle{definition}
\newtheorem{expl}[thm]{Example}

\theoremstyle{definition2}
\newtheorem*{acknow}{Acknowledgments}

\newtheorem*{rem}{Remark}

\numberwithin{equation}{section}

\makeatletter
\renewcommand{\thefigure}{\ifnum \c@section>\z@ \thesection.\fi
 \@arabic\c@figure}
\@addtoreset{figure}{section}
\makeatother

\begin{document}

\def\BB{\mathcal B} \def\N{\mathbb N} \def\PP{\mathcal P}
\def\Homeo{{\rm Homeo}} \def\MM{\mathcal M} \def\A{\mathbb A}
\def\S{\mathbb S} \def\D{\mathbb D} \def\C{\mathbb C}
\def\Z{\mathbb Z} \def\P{\mathbb P} \def\CC{\mathcal C}
\def\AA{\mathcal A} \def\QQ{\mathcal Q} \def\conj{{\rm conj}}
\def\Im{{\rm Im}} 


\begin{abstract}
We prove that, for $n \ge 5$, an interval group associated with a proper quasi-Coxeter element of the Coxeter group of type $D_n$ admits no surjective homomorphism onto the braid group on $n$ strands.
In particular, this provides an alternative proof that such a group is not isomorphic to the Artin group of type $D_n$.
The proof relies on techniques from the theory of mapping class groups, and a large part of the paper provides an exposition of this theory for non-specialists.



\end{abstract}

\maketitle


\section{Introduction}\label{sec1}

Let $S$ be a finite set.
A \emph{Coxeter matrix} over $S$ is a square matrix $M = (m_{s,t})_{s,t \in S}$ indexed by the elements of $S$, with entries in $\N_{\ge 1} \cup \{ \infty\}$, such that $m_{s,s} = 1$ for all $s \in S$ and $m_{s,t} = m_{t,s} \ge 2$ for all distinct $s,t \in S$.
Such a matrix is usually represented by a labeled graph, $\Gamma$, called the \emph{Coxeter graph} of $M$.
The vertex set of $\Gamma$ is $S$, two vertices $s$ and $t$ are joined by an edge whenever $m_{s,t} \ge 3$, and this edge is labeled by $m_{s,t}$ whenever $m_{s,t} \ge 4$.

To a Coxeter graph $\Gamma$, we associate an \emph{Artin group} $A[\Gamma]$ and a \emph{Coxeter group} $W[\Gamma]$.
These groups are defined as follows.
Given two letters $s,t$ and an integer $m \ge 2$, we denote by $\Pi(s,t,m)$ the alternating product $sts \cdots$ of length $m$.
Let $M = (m_{s,t})_{s,t \in S}$ be the Coxeter matrix of $\Gamma$.
Then
\[ 
A [\Gamma] = \langle S \mid \Pi (s,t, m_{s,t}) = \Pi (t,s, m_{s,t})\text{ for } s,t \in S\,,\ s \neq t\,,\ m_{s,t} \neq \infty \rangle \,.
\] 
The group $W [\Gamma]$ is the quotient of $A [\Gamma]$ by the relations $s^2=1$, for all $s \in S$.

The \emph{set of reflections} of $W[\Gamma]$ is $T = \{ w s w^{-1} \mid w \in W[\Gamma]\,,\ s \in S\}$.
Since $T$ generates $W[\Gamma]$, we can consider the word-length $\ell_T: W[\Gamma] \to \N$ with respect to $T$.
The \emph{absolute order} on $W[\Gamma]$ is the partial order $\preceq$ defined by $u \preceq v$ if $\ell_T(v) = \ell_T(u) + \ell_T(u^{-1} v)$.
For $w \in W[\Gamma]$, we set $[1,w] = \{ u \in W[\Gamma] \mid u \preceq w \}$.

The \emph{interval group} associated with an element $w \in W [\Gamma]$ is the group $G([1,w])$ defined by the following presentation.
The generating set of $G([1,w])$ is an abstract set $D = \{ d_u \mid u \in [1,w] \}$ in one-to-one correspondence with $[1,w]$.
The set of defining relations is:
\[
\{ d_u d_v = d_{uv} \mid u,v, uv \in [1,w] \text{ and } \ell_T(u) + \ell_T(v) = \ell_T(uv) \}\,.
\]

Recall that a \emph{Coxeter element} is a product in $W[\Gamma]$ of all the elements of $S$ in some order.
Interval groups were introduced by Bessis \cite{Bes03} in the case where $w=c$ is a Coxeter element.
In this case, the group $G([1,c])$ is called a \emph{Coxeter interval group} or a \emph{dual Artin group}.
These groups play a central role in the study of Artin groups, particularly in connection with the $K(\pi,1)$ conjecture (see \cite{PS21,DPS24}).
We refer to \cite{Pao25} for a discussion of the ``Coxeter interval group'' approach to the study of the $K(\pi,1)$ conjecture. 

Two central questions about interval groups are the following: 
\begin{itemize}
\item[(1)] 
Is $G([1,w])$ isomorphic to $A[\Gamma]$? 
\item[(2)] 
Is $G([1,w])$ a Garside group? 
\end{itemize} 
These two questions are generally related. 
The classical case is when $w=c$ is a Coxeter element and $\Gamma$ is of spherical type (i.e., $W[\Gamma]$ is finite). 
In this case, the answer to both questions is yes (see \cite{Bes03, BW08}). 
Other known cases in which $w=c$ is a Coxeter element and the answer to both questions is yes are universal Coxeter groups (i.e., with $m_{s,t} = \infty$ for all distinct $s,t \in S$) (see \cite{Bes06}) and Coxeter groups of rank $3$ (i.e., with $|S|=3$) (see \cite{DPS24}). 
On the other hand, both questions are open when $\Gamma$ is right-angled (i.e., with $m_{s,t} \in \{2, \infty\}$ for all distinct $s,t \in S$). 
The case where $\Gamma$ is of affine type and $w=c$ is a Coxeter element is more tricky. 
In this case, $G([1,c])$ is known to be isomorphic to $A[\Gamma]$, but it is not always a Garside group. 
However, it always embeds into a Garside group \cite{Dig06, Dig12, McS17, PS21}.

An element $w \in W [\Gamma]$ is called a \emph{quasi-Coxeter element} if $\ell_T (w) = |S|$ and there exists an expression $w=t_1 t_2 \cdots t_n$ of $w$ over $T$ such that $n=|S|$ and $\{t_1, \dots, t_n\}$ generates $W [\Gamma]$.
Every Coxeter element is a quasi-Coxeter element, but the converse does not hold in general.
A quasi-Coxeter element that is not conjugate to a Coxeter element is called a \emph{proper quasi-Coxeter element}.

Quasi-Coxeter elements have been studied primarily in the case where $\Gamma$ is of spherical type.
In this case, they exhibit behaviour similar to that of Coxeter elements (see \cite{Car72, BDSW14, BGRW17}).
For example, the Hurwitz action on the set of the reduced reflection factorisations is transitive.
It is therefore natural to study interval groups associated with quasi-Coxeter elements after studying those associated with Coxeter elements.
Note that there are no proper quasi-Coxeter elements in Coxeter groups of type $A_n$, $B_n$ and $I_2(m)$, hence the first interesting case is $W[D_n]$ with $n \ge 4$, which contains $\lfloor (n-1)/2 \rfloor$ conjugacy classes of proper quasi-Coxeter elements \cite{BGRW17}.

An interval group associated with a (proper) quasi-Coxeter element is called a \emph{(proper) quasi-Coxeter interval group}.
Presentations for all quasi-Coxeter interval groups of finite Coxeter groups have been computed in \cite{BNR23} and \cite{BHNR23a}.
Furthermore, it is shown in \cite{BHNR23b} that a proper quasi-Coxeter interval group of type $D_n$ is not isomorphic to $A[D_n]$, thereby providing a negative answer to Question (1) in this case.
However, it is still unknown whether a proper quasi-Coxeter interval group of type $D_n$ is a Garside group or embeds into a Garside group.

In \cite{CP05}, it is proved that the Artin group $A[D_n]$ is the semidirect product of the braid group $\BB_n$ on $n$ strands with a free group of rank $n-1$ (see also \cite{PV96,All02}).
Let $w$ be a proper quasi-Coxeter element of $W[D_n]$.
Although it is known that $G([1,w])$ is not isomorphic to $A[D_n]$,
it is natural to ask whether a decomposition exists for $G([1,w])$ that is similar to the one for $A[D_n]$.
An affirmative answer would, in particular, make it possible to solve some open combinatorial questions on $G ([1,w])$, such as the word problem.
Unfortunately, a direct consequence of our main result is that no such decomposition exists.
Our main result is the following.

\begin{thm}\label{thm1_1}
Let $n \ge 5$, and let $w$ be a proper quasi-Coxeter element of $W [D_n]$.
Then $G([1,w])$ admits no surjective homomorphism onto the braid group $\BB_n$ on $n$ strands.
\end{thm}

Since $A [D_n]$ admits a surjective homomorphism onto $\BB_n$, this yields a new proof of the following result
whose proof can be found in Theorem 1.3 of \cite{BHNR23b} together with Theorem A of \cite{BNR23}.

\begin{corl}[Baumeister--Holt--Neaime--Rees \cite{BNR23,BHNR23b}]\label{corl1_2}
Let $n \ge 5$, and let $w$ be a proper quasi-Coxeter element of $W [D_n]$.
Then $G([1,w])$ is not isomorphic to $A [D_n]$.
\end{corl}

\begin{rem}
The arguments used in the proof of Theorem \ref{thm1_1} require $n \ge 5$, hence they do not apply to the case $n=4$. 
Actually, we do not known whether Theorem \ref{thm1_1} can be extended to $n=4$.
However, Corollary \ref{corl1_2} is proved in \cite{BNR23,BHNR23b} for all $n \ge 4$, rather than only for $n \ge 5$. 
\end{rem}

The interest of the paper also lies, and perhaps primarily so, in the proof of Theorem \ref{thm1_1} itself, which relies on tools from low-dimensional topology.
More precisely, we make use of several classical results on mapping class groups applied to the mapping class group of the $n$-punctured disc, which is isomorphic to the braid group $\BB_n$.
The use of these methods is new in the study of interval groups, and the authors assume that a potential reader would not be familiar with this theory, hence a large part of the paper (namely, the entirety of Section \ref{sec2}) provides an exposition on mapping class groups.
Once these tools are in place, the proof of Theorem \ref{thm1_1} follows directly from the presentations of quasi-Coxeter interval groups of type $D_n$ given in \cite{BNR23}, together with results on endomorphisms of braid groups by Castel \cite{Cas16} and Chen--Kordek--Margalit \cite{CKM19}.
Section \ref{sec3} is devoted to the proof of Theorem \ref{thm1_1}.

\begin{acknow}

The first author is partially supported by the German Research Foundation
SFB-TRR 358/1 2023 – 491392403.
The second author is partially supported by the French project ``CaGeT'' (ANR-25-CE40-4162) of the ANR.
The IMB receives support from the EIPHI Graduate School (contract ANR-17-EURE-0002).
\end{acknow}


\section{Preliminaries on mapping class groups}\label{sec2}

As pointed out in the introduction, in our proof of Theorem \ref{thm1_1} we use tools from the theory of mapping class groups. 
In these notes we provide the information on these that we shall need and refer to
\cite{FM12} for a detailed account of the subject.

Let $\Sigma$ be a compact oriented surface with or without boundary and let $\PP$ be a finite family of points in the interior of $\Sigma$. 
We denote by $\Homeo^+ (\Sigma, \PP)$ the group of homeomorphisms of $\Sigma$ that preserve orientation, that act as the identity on a neighbourhood of the boundary, and that leave the set $\PP$ fixed as a set.
The \emph{mapping class group} of the pair $(\Sigma, \PP)$, denoted by $\MM (\Sigma, \PP)$, is the group of isotopy classes of elements of $\Homeo^+ (\Sigma, \PP)$.

Two kinds of elements of $\MM (\Sigma, \PP)$ are of interest to us: the \emph{Dehn twists} and the \emph{half-twists}. 
We shall start by introducing the Dehn twists.

Let $\A = [0,1] \times \S^1$ be the standard annulus. 
The \emph{standard Dehn twist} is the homeomorphism $T : \A \to \A$ defined by
\[
T (t,z) = (t, e^{2i t \pi} z)\,.
\]
This homeomorphism is shown in Figure \ref{fig2_1}.

\begin{figure}[ht!]
\begin{center}
\includegraphics[width=4cm]{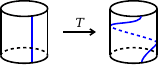}
\caption{Standard Dehn twist.}\label{fig2_1}
\end{center}
\end{figure}

A \emph{simple closed curve} in $(\Sigma, \PP)$ is an embedding $a : \S^1 \to \Sigma$ whose image does not intersect either $\partial \Sigma$ or $\PP$. 
We say that $a$ is  \emph{essential} if its image does not bound any disc embedded in $\Sigma$ that contains $0$ or $1$ elements of $\PP$. 
Let $a : \S^1 \to \Sigma$ be a simple and essential closed curve. 
We choose an embedding $\hat a : \A \to \Sigma$ whose image does not intersect either $\partial \Sigma$ or $\PP$, and such that $\hat a (\frac{1}{2},z) = a(z)$ for each $z \in \S^1$.
We define $T_a \in \Homeo^+ (\Sigma, \PP)$ as follows. 
Suppose that $x \in \Sigma$.
\begin{itemize}
\item[(1)]
If $x$ is within the image of $\hat a$, with $x = \hat a (y)$, then $T_a (x) = \hat a (T(y))$, where $T$ denotes the standard Dehn twist.
\item[(2)]
If $x$ is outside the image of $\hat a$, then $T_a (x) = x$.
\end{itemize}
The \emph{Dehn twist} along $a$, denoted by $\tau_a$, is the element of $\MM (\Sigma, \PP)$ represented by $T_a$ (see Figure~\ref{fig2_2}).

\begin{figure}[ht!]
\begin{center}
\includegraphics[width=10.6cm]{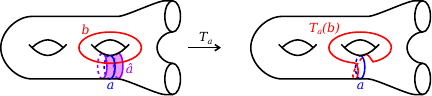}
\caption{Dehn twist.}\label{fig2_2}
\end{center}
\end{figure}

The following result guarantees that $\tau_a$ is well defined as an element of the mapping class group.

\begin{prop}\label{prop2_1}\leavevmode
\begin{itemize}
\item[(1)] 
The definition of $\tau_a$ does not depend on the image of $a$, or on the mapping $a$, or on its orientation, or on the choice of $\hat a$.
\item[(2)]
If $a$ and $b$ are isotopic simple and essential closed curves, then $\tau_a = \tau_b$.
\end{itemize}
\end{prop}

Let $\D = \{ z \in \C \mid |z| \le 1\}$ be the standard disc and $\PP_2 = \{\frac{-1}{2}, \frac{1}{2} \}$.
The \emph{standard half-twist} is the homeomorphism $S \in \Homeo^+ (\D, \PP_2)$ defined by:
\[
S(z) = e^{-2i \pi |z|} z\,.
\]
Note that this homeomorphism interchanges the two points of $\PP_2$.
It is illustrated in Figure \ref{fig2_3}.

\begin{figure}[ht!]
\begin{center}
\includegraphics[width=5.6cm]{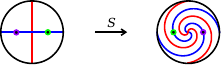}
\caption{Standard half-twist.}\label{fig2_3}
\end{center}
\end{figure}

An \emph{arc} in $(\Sigma, \PP)$ is an embedding $a : [0,1] \to \Sigma$ such that $a(0), a(1) \in \PP$, $a((0,1)) \cap \PP = \emptyset$, and $a([0,1]) \cap \partial \Sigma = \emptyset$. 
Let $a : [0,1] \to \Sigma$ be an arc. 
We choose an embedding $\hat a : \D \to \Sigma$ such that $\hat a (\D) \cap \PP = \{ a(0), a(1) \}$, $\hat a (\D) \cap \partial \Sigma = \emptyset$, and $\hat a (t-\frac{1}{2}) = a(t)$ for each $t \in [0,1]$. 
We define $S_a \in \Homeo^+ (\Sigma, \PP)$ as follows. 
Suppose that $x \in \Sigma$. 
\begin{itemize}
\item[(1)]
If $x$ is in the image of $\hat a$, with $x = \hat a (y)$, then $S_a (x) = \hat a (S(y))$, where $S$ denotes the standard half-twist.
\item[(2)]
If $x$ is outside the image of $\hat a$, then $S_a (x) = x$.
\end{itemize}
The \emph{half-twist} along $a$, denoted by $\sigma_a$, is the element of $\MM (\Sigma, \PP)$ represented by $S_a$ (see Figure~\ref{fig2_4}).

\begin{figure}[ht!]
\begin{center}
\includegraphics[width=10.2cm]{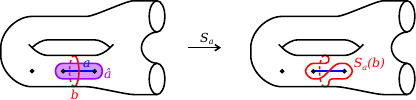}
\caption{Half-twist.}\label{fig2_4}
\end{center}
\end{figure}

Note that $S_a$ interchanges the two extremities of $a$, and, as for Dehn twists, the element $\sigma_a$ is well defined, due to the following result.

\begin{prop}\label{prop2_2}\leavevmode
\begin{itemize}
\item[(1)] 
The definition of $\sigma_a$ only depends on the image of $a$, and not on the mapping  $a$, or on its orientation, or on the choice of $\hat a$.
\item[(2)]
If $a$ and $b$ are two isotopic arcs, then $\sigma_a = \sigma_b$.
\end{itemize}
\end{prop}

Let $a : [0,1] \to \Sigma$ be an arc. 
As before, we choose an embedding $\hat a : \D \to \Sigma$ such that $\hat a (\D) \cap \PP = \{ a(0), a(1) \}$, $\hat a (\D) \cap \partial \Sigma = \emptyset$, and $\hat a (t-\frac{1}{2}) = a(t)$ for each $t \in [0,1]$. 
We denote by $b : \S^1 \to \Sigma$ the mapping defined by $b(z) = \hat a (z)$ (see Figure \ref{fig2_5}). 
It is clear that $b$ is a simple and essential closed curve if $\Sigma$ is not a sphere with fewer than $4$ marked points. 
It is called a \emph{regular boundary} of $a$. 

\begin{figure}[ht!]
\begin{center}
\includegraphics[width=2.6cm]{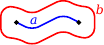}
\caption{Regular boundary of an arc.}\label{fig2_5}
\end{center}
\end{figure}

The elements $\sigma_a$ and $\tau_b$ are connected by the following relation.

\begin{prop}\label{prop2_3}
Under the above hypotheses, $\tau_b = \sigma_a^2$.
\end{prop}

Recall that the \emph{braid group} $\BB_n$ on $n$ strands is the Artin group $\BB_n = A [A_{n-1}]$ of the Coxeter graph $A_{n-1}$ shown in Figure \ref{fig2_6}. 
The standard generators of $\BB_n$ will always be denoted by $t_1, \dots, t_{n-1}$ as seen in the figure.

\begin{figure}[ht!]
\begin{center}
\includegraphics[width=4.4cm]{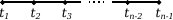}
\caption{Coxeter graph $A_{n-1}$.}\label{fig2_6}
\end{center}
\end{figure}

The example of a mapping class group that interests us in this paper is the one where $\Sigma = \D = \{ z \in \C \mid |z| \le 1\}$ is the standard disc and $\PP_n = \{p_1, \dots, p_n\}$ is a collection of $n$ points in the interior of $\D$ positioned as in Figure \ref{fig2_7}. 
Let $e_1, \dots, e_{n-1}$ be the arcs sketched in Figure \ref{fig2_7}. 
Then:

\begin{figure}[ht!]
\begin{center}
\includegraphics[width=5.6cm]{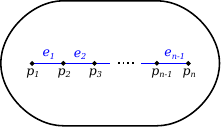}
\caption{Marked disc.}\label{fig2_7}
\end{center}
\end{figure}

\begin{thm}[Artin \cite{Art47}]\label{thm2_4}
Let $n \ge 1$ and $t_1, \dots, t_{n-1}$ be the standard generators of the braid group $\BB_n$.
We have an isomorphism $\BB_n \to \MM (\D, \PP_n)$ that maps $t_i$ to $\sigma_{e_i}$ for each $1 \le i \le n-1$.
\end{thm}

\begin{expl}\label{expl2_5}
Applying Theorem \ref{thm2_4} with $n=1$ we deduce that $\MM (\D, \PP_1) = \{ 1\}$. 
It is also well-known that $\MM (\D, \emptyset) = \{ 1\}$ (see \cite[Lemma 2.1]{FM12}). 
Applying it with $n =2$ we deduce that $\MM (\D, \PP_2) \simeq \Z$ is an infinite 
cyclic group generated by the half-twist $\sigma_{e_1}$.
\end{expl}

\begin{expl}\label{expl2_6}
We recall that the \emph{cylinder} is the surface $\A = \S^1 \times [0,1]$ of genus $0$ with two boundary components and that the \emph{pair of pants} is the surface $\P$ of genus $0$ with three boundary components. 
Other mapping class groups of small surfaces that interest us are that of the cylinder with $0$ or $1$ marked points, and that of the pair of pants with $0$ marked points.
The following result is found in Section 2.2.2 of \cite{FM12}.
\end{expl}

\begin{prop}\label{prop2_7}\leavevmode
\begin{itemize}
\item[(1)] 
The group $\MM (\A, \emptyset)$ is infinite cyclic generated by $\tau_a$, where $a$ is the simple closed curve illustrated on the left hand portion of Figure \ref{fig2_8}.
\item[(2)]
Let $p$ be a point in the interior of $\A$. 
Then $\MM (\A, \{ p \})$ is a free abelian group of rank $2$ generated by $\tau_{a_1}$ and $\tau_{a_2}$, where $a_1$ and $a_2$ are the simple closed curves illustrated in the central portion of Figure \ref{fig2_8}.
\item[(3)]
The group $\MM (\P, \emptyset)$ is a free abelian group of rank three generated by $\tau_{a_1}$, $\tau_{a_2}$ and $\tau_{a_3}$, where $a_1$, $a_2$ and $a_3$ are the simple closed curves sketched in the right hand portion of Figure \ref{fig2_8}.
\end{itemize}
\end{prop}

\begin{figure}[ht!]
\begin{center}
\includegraphics[width=8.4cm]{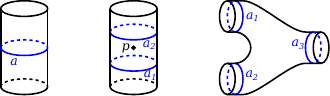}
\caption{Mapping class groups of the cylinder with  $0$ or $1$ marked points and of the pair of pants with $0$ marked points.}\label{fig2_8}
\end{center}
\end{figure}

The isotopy class of a simple and essential closed curve $a$ is denoted by $[a]$.
We denote by  $\CC (\Sigma, \PP)$ the set of isotopy classes of simple and essential closed curves on $(\Sigma, \PP)$. 
Note that, by Proposition~\ref{prop2_1}, if $u \in \CC (\Sigma, \PP)$ and $a,b \in u$, then $\tau_a = \tau_b$. 
So we can legitimately set $\tau_u = \tau_a$, where $a$ is any element of $u$.

We have an action of $\MM (\Sigma, \PP)$ on $\CC (\Sigma, \PP)$ defined as follows. 
Suppose that $f \in \MM (\Sigma, \PP)$ and $u \in \CC (\Sigma, \PP)$. 
We choose a homeomorphism $F \in \Homeo^+ (\Sigma, \PP)$ that represents $f$, a curve $a \in u$, and we set $f(u) = [F(a)]$. 

We define the \emph{intersection number} of two classes $u,v \in \CC (\Sigma, \PP)$ to be:
\[
i(u,v) = \min\{ | a \cap b| \mid a \in u \text{ and } b \in v\}\,.
\]
The following proposition brings together several fundamental results for our study.
Their proofs can be found in \cite[Section 3.3]{FM12}.

\begin{prop}\label{prop2_8}\leavevmode
\begin{itemize}
\item[(1)]
Suppose that $f \in \MM (\Sigma, \PP)$ and $u \in \CC (\Sigma, \PP)$.
Then $f \tau_u f^{-1} = \tau_{f(u)}$.
\item[(2)]
Suppose that $u,v \in \CC (\Sigma, \PP)$.
We have $\tau_u = \tau_v$ if and only if $u=v$.
\item[(3)]
Suppose that $u,v \in \CC (\Sigma, \PP)$.
We have $\tau_u \tau_v = \tau_v \tau_u$ if and only if $i(u,v)=0$.
\end{itemize}
\end{prop}

We say that two simple and essential closed curves $a$ and $b$ on $(\Sigma, \PP)$ are in \emph{minimal position} if $i ([a], [b]) = |a \cap b|$. 
Further, we say that $a$ and $b$ \emph{co-bound a disc} if there exists a disc $D$ embedded in the interior of  $\Sigma$, bounded by the union of an arc of $a$ and an arc of $b$, and not intersecting $\PP$ (see Figure \ref{fig2_9}). 
The following result yields a very useful criterion for deciding whether two simple and essential closed curves are in minimal position.

\begin{figure}[ht!]
\begin{center}
\includegraphics[width=3.2cm]{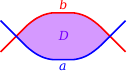}
\caption{Disc co-bounded by $a$ and $b$.}\label{fig2_9}
\end{center}
\end{figure}

\begin{thm}[Epstein \cite{Eps66}]\label{thm2_9}
Suppose that $|\PP| \ge 3$ if $\Sigma$ has genus $0$ and that $|\PP| \ge 1$ if $\Sigma$ has genus $1$. 
Let $a$ and $b$ be two simple and essential closed curves.
Then $a$ and $b$ are in minimal position if and only if they do not co-bound any disc.
\end{thm}

The following result is proved in \cite{FLP79} (see also \cite[Proposition 3.4]{FM12}).

\begin{thm}\label{thm2_10}
Suppose that $|\PP| \ge 3$ if $\Sigma$ has genus $0$ and that $|\PP| \ge 1$ if $\Sigma$ has genus $1$. 
Suppose that $u_1, \dots, u_p, v, w \in \CC (\Sigma, \PP)$ and $k_1, \dots, k_p \in \Z$ such that $i(u_i, u_j) = 0$ for each $1 \le i,j \le p$. 
Set $g=\tau_{u_1}^{k_1} \cdots \tau_{u_p}^{k_p}$. 
Then
\[
\left| i (g(w), v) - \sum_{i=1}^p |k_i| \, i(u_i, v)\, i(u_i,w) \right| \le i (v,w)\,.
\]
\end{thm}

Setting $v = w$ we derive:

\begin{corl}\label{corl2_11}
Suppose that $|\PP | \ge 3$ if $\Sigma$ has genus $0$ and that $|\PP| \ge 1$ if $\Sigma$ has genus $1$. 
Suppose that $u_1, \dots, u_p, v \in \CC (\Sigma, \PP)$ and $k_1, \dots, k_p \in \Z$ such that $i(u_i, u_j) = 0$ for each $1 \le i,j \le p$. 
Set $g=\tau_{u_1}^{k_1} \cdots \tau_{u_p}^{k_p}$. 
Then
\[
i(g(v), v) = \sum_{i=1}^p |k_i| \, i(u_i,v)^2\,.
\]
\end{corl}

\begin{expl}\label{expl2_12}
One application of the above to the mapping class group $\MM (\D, \PP_n) = \BB_n$ for $n \ge 4$ that we shall use in our proof is the following.
Let $a_1, a_2$ and $b$ be the simple closed curves sketched in Figure~\ref{fig2_10}. 
We observe that $a_1$ and $b$ do not co-bound any disc, hence, by Theorem \ref{thm2_9}, $i ([a_1], [b])=2$.
Further, $i ([a_2], [b]) = 2$.
Suppose that $k_1, k_2 \in \Z$ and $g = \tau_{a_1}^{k_1} \tau_{a_2}^{k_2}$.
Applying Corollary \ref{corl2_11} we have $i (g ([b]), [b]) = 4 (|k_1| + |k_2|)$.
Hence, by Proposition \ref{prop2_8}, $\tau_{g ([b])}$ and $\tau_{b}$ commute if and only if $k_1 = k_ 2 = 0$.
\end{expl}

\begin{figure}[ht!]
\begin{center}
\includegraphics[width=7.2cm]{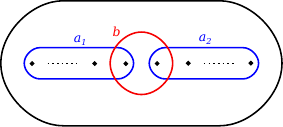}
\caption{The curves $a_1, a_2$ and $b$ from Example \ref{expl2_12}.}\label{fig2_10}
\end{center}
\end{figure}

We say that a finite set $\AA \subseteq \CC (\Sigma, \PP)$ \emph{describes a triangle} if there exist $u,v,w \in \AA$ such that $i (u,v) \neq 0$, $i (u,w) \neq 0$ and $i (v,w) \neq 0$. 
In the opposite case, we say that $\AA$ \emph{does not describe any triangle}. 
It follows immediately from Lemme 2.9 of \cite{FM12} (or see \cite[Proposition 2.1.3]{Cas16}) that:

\begin{thm}\label{thm2_13}
Let $a_1, \dots, a_p$ be simple closed curves such that $a_i$ and $a_j$ are in minimal position for all $i,j \in \{1, \dots, p\}$ with $i \neq j$.
Suppose that $f \in \MM(\Sigma, \PP)$. 
Suppose that $\AA = \{ [a_1], \dots, [a_p] \}$ does not describe any triangle and that $f([a_i]) = [a_i]$ for each $i \in \{1, \dots, p\}$. 
Then there exists $F \in \Homeo^+ (\Sigma, \PP)$ representing $f$ such that $F(a_i) = a_i$ for each $i \in \{1, \dots, p\}$.
\end{thm}

\begin{expl}\label{expl2_14}
Now we shall show how to use the above results to determine the centraliser of
$\{ t_1, t_2, t_4, t_5, t_6\}$ in the braid group $\BB_7$. 
For this we identify $\BB_7$ with $\MM (\D, \PP_7)$ and consider the simple closed curves $a_1$, $a_2$ and $d$ sketched in Figure \ref{fig2_11}.
\end{expl}

\begin{figure}[ht!]
\begin{center}
\includegraphics[width=8cm]{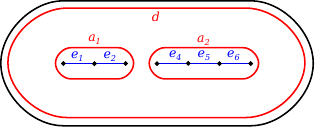}
\caption{Centraliser of  $\{ t_1, t_2, t_4, t_5, t_6\}$ in $\BB_7$.}\label{fig2_11}
\end{center}
\end{figure}

\begin{lem}\label{lem2_15}
The centraliser of $\{ t_1, t_2, t_4, t_5, t_6\}$ in $\BB_7$ is a free abelian group of rank $3$ generated by $\tau_{a_1}$, $\tau_{a_2}$ and $\tau_{d}$.
\end{lem}

\begin{rem}
Recall that the centre of $\BB_7$ is the infinite cyclic group generated by $\Delta^2 = (t_1 t_2 \cdots t_6)^7$, where $\Delta$ denotes the Garside element. 
If we interpret $\BB_7$ as $\MM (\D, \PP_7)$, then $\Delta^2 = \tau_d$.
Further, $\tau_{a_1}$ is a generator of the centre of $\langle t_1, t_2 \rangle$ and $\tau_{a_2}$ is a generator of the centre of $\langle t_4, t_5, t_6 \rangle$. 
More precisely, $\tau_{a_1} = (t_1 t_2)^3$ and $\tau_{a_2} = (t_4 t_5 t_6)^4$.
\end{rem}

In the proof of Lemma \ref{lem2_15} we shall use the following construction. 
Let $(\Sigma, \PP)$ be a surface equipped with a finite set of points. 
Let $\Omega$ be a sub-surface of $\Sigma$ such that $\PP \cap \partial \Omega = \emptyset$ (see Figure~\ref{fig2_12}).
Set $\QQ = \PP \cap \Omega$. 
We have an embedding $\hat \iota_\Omega : \Homeo^+ (\Omega, \QQ) \to \Homeo^+ (\Sigma, \PP)$ defined as follows.
If $F \in \Homeo^+ (\Omega, \QQ)$, then $\hat \iota_\Omega (F)$ is the homeomorphism of $\Sigma$ that coincides with $F$ on $\Omega$ and that is the identity outside $\Omega$. 
We verify easily that $\hat \iota_\Omega$ induces a homomorphism $\iota_\Omega : \MM (\Omega, \QQ) \to \MM (\Sigma, \PP)$. 
This homomorphism is often injective, but not always.
We refer to \cite{PR00} for a detailed study of this homomorphism.

\begin{figure}[ht!]
\begin{center}
\includegraphics[width=4.6cm]{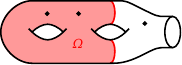}
\caption{Sub-surface.}\label{fig2_12}
\end{center}
\end{figure}

\begin{proof}[Proof of Lemma \ref{lem2_15}]
It is clear that $\tau_{a_1}$, $\tau_{a_2}$ and $\tau_d$ commute with $t_i = \sigma_{e_i}$ for each $i \in \{1, 2, 4, 5, 6 \}$. 
Further, the fact that $\tau_{a_1}$, $\tau_{a_2}$ and $\tau_d$ generate a free abelian group of rank $3$ is a classical result (see \cite[Lemma 3.17]{FM12}).
It remains to prove that, if $g \in \BB_7 = \MM (\D, \PP_7)$ commutes with $t_i = \sigma_{e_i}$ for each $i \in \{1,2,4,5,6\}$, then $g \in \langle \tau_{a_1}, \tau_{a_2}, \tau_d \rangle$.

Let $g$ be an element of $\BB_7 = \MM (\D, \PP_7)$ that commutes with $t_i = \sigma_{e_i}$  for each $i \in \{1, 2, 4, 5, 6\}$. 
Let $b_1, b_2, b_4, b_5, b_6$ be the simple closed curves sketched in Figure \ref{fig2_13}. 
By Proposition \ref{prop2_3}, we have $t_i^2=\sigma_{e_i}^2 = \tau_{b_i}$, hence $g$ commutes with $\tau_{b_i}$ for each $i \in \{1, 2, 4, 5, 6\}$. 
By Proposition \ref{prop2_8}\,(1) it follows that $ g \tau_{[b_i]} g^{-1} = \tau_{ g ([b_i])} = \tau_{[b_i]}$, hence, by Proposition \ref{prop2_8}\,(2), $g ([b_i]) = [b_i]$, for each $i \in \{1, 2, 4, 5, 6\}$. 
Further, by Theorem \ref{thm2_9}, the curves $b_i$ and $b_j$ are in minimal position for all $i,j \in \{1, 2, 4, 5, 6\}$ with $i \neq j$. 
We observe further that $\{ [b_1], [b_2], [b_4], [b_5], [b_6]\}$ does not describe any triangle. 
We deduce from Theorem \ref{thm2_13} that there exists $G \in \Homeo^+ (\D, \PP_7)$ representing $g$ such that $G(b_i) = b_i$ for each $i \in \{1, 2, 4, 5, 6\}$.

\begin{figure}[ht!]
\begin{center}
\includegraphics[width=8cm]{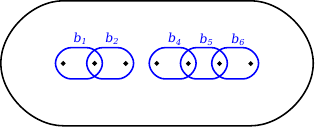}
\caption{The curves $b_1, b_2, b_4,b_5,b_6$.}\label{fig2_13}
\end{center}
\end{figure}

Suppose that $i \in \{1, 2, 4, 5, 6\}$. 
If $G$ reverses the orientation of $b_i$, then $G$ permutes the two connected components of $\D \setminus b_i$. 
But this is not possible because one of the two connected components contains $\partial \D$, which is fixed pointwise by $G$. 
We deduce that $G$ preserves the orientation of $b_i$. 
Hence we can suppose that $G$ acts as the identity on a neighbourhood of $b_i$ for each $i \in \{1, 2, 4, 5, 6\}$.

The set $B = b_1 \cup b_2 \cup b_4 \cup b_5 \cup b_6$ divides $\D$ into $9$ sub-surfaces (see Figure \ref{fig2_14}): 
\begin{itemize}
\item
$7$ discs $D_1, \dots, D_7$ such that, for each $i \in \{1, \dots, 7\}$, $\PP_7 \cap D_i = \{p_i\}$,
\item
a disc $D_0$ such that $D_0 \cap \PP_7 = \emptyset$,
\item
a pair of pants $P$ such that $P \cap \PP_7 = \emptyset$ and of which the three boundary components are isotopic to $a_1$, $a_2$ and $d$.
\end{itemize}
As explained above, the inclusion $D_i \subseteq \D$ induces a homomorphism $\iota_{D_i} : \MM (D_i, \{ p_i\}) \to \MM (\D, \PP_7)$, the inclusion $D_0 \subseteq \D$ induces a homomorphism $\iota_{D_0} : \MM (D_0, \emptyset) \to \MM (\D, \PP_7)$,  and the inclusion $P \subseteq \D$ induces a homomorphism $\iota_P : \MM (P, \emptyset) \to \MM (\D, \PP_7)$. 
Since $G$ acts as the identity on a neighbourhood of $B$, we can consider $g_i \in \MM (D_i, \{ p_i\})$ represented by the restriction of $G$ to $D_i$ for each $i \in \{1, \dots, 7\}$, $g_0 \in \MM (D_0, \emptyset)$ represented by the restriction of $G$ to $D_0$, and $h \in \MM(P, \emptyset)$ represented by the restriction of $G$ to $P$. 
Therefore, by construction, 
\[
g = \iota_P (h)\, \iota_{D_0} (g_0) \, \iota_{D_1} (g_1)\, \iota_{D_2} (g_2) \cdots \iota_{D_7} (g_7)\,.
\]
By Example \ref{expl2_5}, we have $g_i = 1$ for each $i \in \{0,1, \dots, 7\}$, and by Proposition \ref{prop2_7}\,(3), $\iota_P (h) \in \langle \tau_{a_1}, \tau_{a_2}, \tau_d \rangle$. 
We conclude that $g \in \langle \tau_{a_1}, \tau_{a_2}, \tau_d \rangle$. 
\end{proof}

\begin{figure}[ht!]
\begin{center}
\includegraphics[width=8cm]{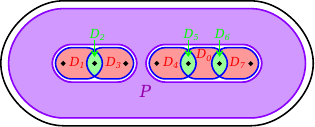}
\caption{Decomposition of $\D$ into sub-surfaces.}\label{fig2_14}
\end{center}
\end{figure}


\section{Proof of Theorem \ref{thm1_1}}\label{sec3}

As indicated in the introduction, once the tools of Section \ref{sec2} are in place, the proof of Theorem \ref{thm1_1} follows directly from the presentations of quasi-Coxeter interval groups of type $D_n$ given in \cite{BNR23}, together with results of Castel \cite{Cas16} and Chen--Kordek--Margalit \cite{CKM19} on endomorphisms of braid groups.
We begin by recalling these two results.

Assume that $n \ge 4$ and $2 \le p \le n-2$.
Let $\Omega_{n,p}$ be the Coxeter graph shown in Figure \ref{fig3_1}.
We denote by $s_1, \dots, s_{p-1}, s_p, s_p', s_{p+1}, \dots, s_{n-1}$ the standard generators of $A [\Omega_{n,p}]$, as illustrated in Figure \ref{fig3_1}.
Let $Q_{n,p}$ be the quotient of $A [\Omega_{n,p}]$ by the relation $[s_{p-1}, s_p^{-1} s_{p+1} s_p' s_{p+1}^{-1} s_p]=1$.
The following result is proved in \cite[Theorem 5.1]{BNR23}.

\begin{figure}[ht!]
\begin{center}
\includegraphics[width=6cm]{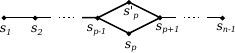}
\caption{Coxeter graph $\Omega_{n,p}$.}\label{fig3_1}
\end{center}
\end{figure}

\begin{thm}[Baumeister--Neaime--Rees \cite{BNR23}]\label{thm3_1}
Let $n \ge 4$, and let $w$ be a proper quasi-Coxeter element of $W[D_n]$.
Then there exists $p \in \{2, \dots, n-2\}$ such that $G([1,w])$ is isomorphic to $Q_{n,p}$.
\end{thm}

A homomorphism $\varphi : G \to H$ is \emph{cyclic} if its image is a cyclic subgroup of $H$. 
Recall that $t_1, \dots, t_{n-1}$ denote the standard generators of $\BB_n$ as shown in Figure \ref{fig2_6}.
It is easily seen that a homomorphism $\varphi : \BB_n \to H$ is cyclic if and only if $h \in H$ exists such that $\varphi (t_i) = h$ for each $1 \le i \le n-1$. 
For each element $h$ in a group $H$ we denote by $\conj_h : H \to H$, $g \mapsto h g h^{-1}$, the conjugation map by $h$. 
Two homomorphisms $\varphi_1, \varphi_2 : G \to H$ are called \emph{conjugate} if $h \in H$ exists such that $\varphi_2 = \conj_h \circ \varphi_1$.

Recall that the \emph{Garside element} of $\BB_n$ is
\[
\Delta = (t_1 \cdots t_{n-1}) (t_1 \cdots t_{n-2}) \cdots (t_1 t_2) t_1 \,.
\]
Recall too that the centre of $\BB_n$ is the infinite cyclic group generated by $\Delta^2$. 
We verify easily that, for all $\varepsilon \in \{ \pm 1\}$ and $k \in \Z$, 
the map
\[
t_i \mapsto t_i^\varepsilon \Delta^{2k}\,, \quad 1 \le i \le n-1\,,
\]
defines a homomorphism $\alpha_{\varepsilon,k} : \BB_n \to \BB_n$.

The following result is proved in \cite{Cas16} for $n \ge 6$ and in \cite{CKM19} for $n \ge 5$.

\begin{thm}[Castel \cite{Cas16}, Chen--Kordek--Margalit \cite{CKM19}]\label{thm3_2}
Suppose that $n \ge 5$. 
Let $\varphi : \BB_n \to \BB_n$ be a homomorphism. 
Then one of the two following possibilities holds:
\begin{itemize}
\item[(1)]
$\varphi$ is cyclic.
\item[(2)]
For some $\varepsilon \in \{ \pm 1\}$ and $k \in \Z$, $\varphi$ is conjugate to $\alpha_{\varepsilon,k}$.
\end{itemize}
\end{thm}

We now turn to the proof of Theorem \ref{thm1_1}.

\begin{proof}[Proof of Theorem \ref{thm1_1}]
We shall prove Theorem \ref{thm1_1} by contradiction.
So, we take a proper quasi-Coxeter element $w$ of $W [D_n]$ and we assume the existence of a surjective homomorphism $\varphi : G ([1,w]) \to \BB_n$.
By Theorem \ref{thm3_1}, we may assume that $p \in \{2, \dots, n-2\}$ exists such that $G([1,w]) = Q_{n,p}$.

We observe easily from the defining relations of $Q_{n,p}$ and $\BB_n$ the existence of two homomorphisms $\kappa, \kappa' : \BB_n \to Q_{n,p}$ defined by:
\begin{gather*}
\kappa (t_i) = \kappa' (t_i) =  s_i \text{ for } i \in \{ 1, \dots, p-1, p+1, \dots, n-1\}\,,\\
\kappa (t_p) = s_p \text{ and } \kappa' (t_p) = s_p'\,.
\end{gather*}
Hence we can consider the homomorphisms $\psi = \varphi \circ \kappa$ and $\psi' = \varphi \circ \kappa'$ from $\BB_n$ to $\BB_n$.

\begin{claim}\label{claim1}
Neither $\psi$ nor $\psi'$ is cyclic.
\end{claim}

\begin{proof}
Suppose that $\psi$ is cyclic. 
There exists $h \in \BB_n$ such that $\psi (t_i) = \varphi (s_i) = h$ for each $i \in \{1, \dots, p-1, p, p+1, \dots, n-1 \}$. 
Then
\[
\varphi (s_p') =
\varphi (s_{p-1} s_p' s_{p-1} s_p'^{-1} s_{p-1}^{-1}) =
h \, \varphi (s_p') \, h \, \varphi (s_p')^{-1} h^{-1} =
\varphi (s_p s_p' s_p s_p'^{-1} s_p^{-1}) =
\varphi (s_p) = h \,.
\]
Hence $\Im (\psi) = \langle h \rangle$ which is a contradiction, since $\varphi$ is surjective and $\BB_n$ is not cyclic. 
We prove in the same way that $\psi'$ is not cyclic.
\end{proof}

It follows from Theorem \ref{thm3_2} that there exist $\varepsilon \in \{ \pm 1\}$ and $k \in \Z$ such that $\psi$ is conjugate to $\alpha_{\varepsilon,k}$. 
Similarly, there exist $\mu \in \{ \pm 1\}$ and $l \in \Z$ such that $\psi'$ is conjugate to $\alpha_{\mu, l}$. 
Hence we can assume without loss of generality that $\psi = \alpha_{\varepsilon,k}$ and there exists $g \in \BB_n$ such that $\psi' = \conj_g \circ \alpha_{\mu, l}$.

\begin{claim}\label{claim2}
We have $\mu=\varepsilon$ and $l=k$.
\end{claim}

\begin{proof}
Let $\zeta : \BB_n \to \Z$ be the homomorphism that maps $t_i$ to $1$ for each $i \in \{1, \dots, n-1\}$. 
We have $\zeta (\Delta^2) = n(n-1)$ and
\[
\varphi (s_1) = \psi (t_1) =  t_1^{\varepsilon} \Delta^{2k} = \psi' (t_1) = g t_1^\mu \Delta^{2l} g^{-1}\,, 
\]
hence $\zeta (\varphi (s_1)) = \varepsilon + k n (n-1) = \mu + l n (n-1)$. 
This equality is only possible if $l=k$ and $\mu = \varepsilon$.
\end{proof}

\begin{claim}\label{claim3}
We have $k = l = 0$.
\end{claim}

\begin{proof}
Recall the homomorphism $\zeta : \BB_n \to \Z$ that maps $t_i$ to $1$ for each $i \in \{1, \dots, n-1\}$. 
We have $\zeta (\varphi (s_i)) = \varepsilon + k n (n-1)$ for each $i \in \{1, \dots, n-1\}$ and $\zeta (\varphi (s_p'))= \varepsilon + k n (n-1)$, hence the image of $\zeta \circ \varphi$ is $(\varepsilon + k n (n-1)) \Z$. 
Further, since $\zeta$ and $\varphi$ are surjective, the image of $\zeta \circ \varphi$ is $\Z$. 
It follows that $(\varepsilon + k n (n-1)) = \pm 1$ which implies that $k=0$.
\end{proof}

From now on we identity $\BB_n$ with the mapping class group $\MM (\D, \PP_n)$.

\begin{claim}\label{claim4}\leavevmode
\begin{itemize}
\item[(1)]
If $3 \le p \le n-3$, then $g \in \langle \tau_{a_1}, \tau_{a_2}, \tau_d \rangle$, where $a_1$, $a_2$ and $d$ are the simple closed curves shown in Figure \ref{fig3_2}\,(a).
\item[(2)]
If $p=2$, then $g \in \langle \sigma_{e_1}, \tau_{a_2}, \tau_d \rangle$, where $a_2$ and $d$ are the simple closed curves shown in Figure~\ref{fig3_2}\,(b).
\item[(3)]
If $p=n-2$, then $g \in \langle \tau_{a_1}, \sigma_{e_{n-1}}, \tau_d \rangle$, where $a_1$ and $d$ are the simple closed curves shown in Figure \ref{fig3_2}\,(c).
\end{itemize}
\end{claim}

\begin{figure}[ht!]
\begin{center}
\begin{tabular}{ccc}
\includegraphics[width=5.3cm]{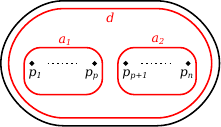} &
\includegraphics[width=4.5cm]{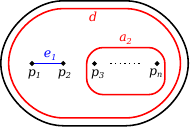} &
\includegraphics[width=4.5cm]{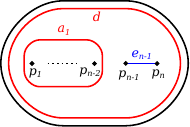}\\
(a) & (b) & (c)
\end{tabular}
\caption{Centraliser of $\{t_1, \dots, t_{p-1}, t_{p+1}, \dots, t_{n-1} \}$.}\label{fig3_2}
\end{center}
\end{figure}

\begin{proof}
We have $\varphi (s_i) = \psi (t_i) = t_i^\varepsilon = \psi' (t_i) = g t_i^\varepsilon g^{-1}$ for each $i \in \{1, \dots, p-1, p+1, \dots, n-1 \}$, hence $g$ is in the centraliser of $\{t_1, \dots, t_{p-1}, t_{p+1}, \dots,  t_{n-1} \}$. 
We show exactly as in Lemma \ref{lem2_15} that the centraliser of $\{t_1, \dots, t_{p-1}, t_{p+1}, \dots, t_{n-1} \}$ is: 
\begin{itemize}
\item
$\langle \tau_{a_1}, \tau_{a_2}, \tau_d \rangle$, if $3 \le p \le n-3$,
\item
$\langle \sigma_{e_1}, \tau_{a_2}, \tau_d \rangle$, if $p = 2$,
\item
$\langle \tau_{a_1}, \sigma_{e_{n-1}}, \tau_d \rangle$, if $p = n-2$.
\end{itemize}
\end{proof}

\begin{claim}\label{claim5}
We have $\varphi (s_p') = t_p^\varepsilon$.
\end{claim}

\begin{proof}
The proof splits into three cases: the case where $3 \le p \le n-3$, the case where $p=2$, and the case where $p=n-2$.

{\it Case 1: $3 \le p \le n-3$.} 
By Claim \ref{claim4}, there exists $g \in \langle \tau_{a_1}, \tau_{a_2}, \tau_d \rangle$ such that $\varphi (s_p') = g t_p^{\varepsilon} g^{-1}$. 
Since $\tau_d$ is central in $\BB_n = \MM (\D, \PP_n)$, we can assume that $g \in \langle \tau_{a_1}, \tau_{a_2} \rangle$, that is, that there exist $k_1, k_2 \in \Z$ such that $g = \tau_{a_1}^{k_1} \tau_{a_2}^{k_2}$. 
Since $s_p$ and $s_p'$ commute, $\varphi (s_p) = t_p^\varepsilon$ and $\varphi (s_p') = g t_p^\varepsilon g^{-1}$ commute, hence $t_p^2$ and $g t_p^2 g^{-1}$ commute. 
By Proposition \ref{prop2_3}, we have $t_p^2 = \tau_{b_p}$, where $b_p$ is a regular boundary of $e_p$ (see Figure \ref{fig3_3}). 
By Theorem \ref{thm2_9}, the curves $a_1$ and $b_p$ are in minimal position, hence  $i([a_1], [b_p]) = 2$. 
Similarly, $a_2$ and $b_p$ are in minimal position, hence $i ([a_2], [b_p]) = 2$. 
By Corollary \ref{corl2_11}, $i ([b_p], g ([b_p])) = 4(|k_1| + |k_2|)$, by Proposition \ref{prop2_8}\,(1), $g t_p^2 g^{-1} = \tau_{g ([b_p])}$, and by Proposition \ref{prop2_8}\,(3), $\tau_{[b_p]}$ and $\tau_{g ([b_p])}$ commute if and only if $i ([b_p], g ([b_p]))=0$. 
We conclude that $k_1 = k_2 = 0$, hence that $g = 1$ and $\varphi (s_p') = t_p^\varepsilon$.

\begin{figure}[ht!]
\begin{center}
\includegraphics[width=7.2cm]{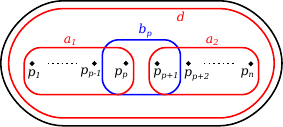}
\caption{The curves $a_1$, $a_2$ and $b_p$.}\label{fig3_3}
\end{center}
\end{figure}

{\it Case 2: $p=2$}.
By Claim \ref{claim4}, there exists $g \in \langle \sigma_{e_1}, \tau_{a_2}, \tau_d \rangle$ such that $\varphi (s_p') = g t_p^{\varepsilon} g^{-1}$. 
Since $\tau_d$ is central in $\BB_n = \MM (\D, \PP_n)$, we can assume that $g \in \langle \sigma_{e_1}, \tau_{a_2} \rangle$, that is, that there exist $k_1, k_2 \in \Z$ such that $g = \sigma_{e_1}^{k_1} \tau_{a_2}^{k_2}$. 

We begin by proving that $k_1$ is even by contradiction.
So, suppose that $k_1$ is odd, that is, $k_1 = 2 l_1 +1$ with $l_1 \in \Z$. 
Denote by $b_1$ a regular boundary of $e_1$ and by $b_2$ a regular boundary of $e_2$. 
Further, we choose a simple closed curve $c$ such that $[c] = \sigma_{e_1} ([b_2])$, as shown in Figure \ref{fig3_4}. 
Recall that, by Proposition \ref{prop2_3}, $t_1^2 = \sigma_{e_1}^2 = \tau_{b_1}$ and $t_2^2 = \sigma_{e_2}^2 = \tau_{b_2}$. 
Set $h = \tau_{b_1}^{l_1} \tau_{a_2}^{k_2}$. 
Then, by Proposition \ref{prop2_8}\,(1), 
\[
g \tau_{b_2} g^{-1} = h \sigma_{e_1} \tau_{b_2} \sigma_{e_1}^{-1} h^{-1} = h \tau_{\sigma_{e_1} ([b_2])} h^{-1} = h \tau_c h^{-1} = \tau_{h([c])}\,.
\]
Since $s_2$ and $s_2'$ commute, $\varphi (s_2) = t_2^\varepsilon$ and $\varphi(s_2') = g t_2^\varepsilon g^{-1}$ commute, hence $t_2^2 = \tau_{b_2}$ and $g t_2^2 g^{-1} = g \tau_{b_2} g^{-1} = \tau_{h ([c])}$ commute. 
Applying Theorem \ref{thm2_9} we have
\[
i([b_1],[b_2]) = i([b_1], [c]) = i([b_2], [a_2]) = i([b_2,c]) = i([c], [a_2]) = 2\,.
\]
Then, by Theorem \ref{thm2_10},
\[
4 (|l_1| + |k_2|)| -2 \le i(h([c]), [b]) \le 4 (|l_1| + |k_2|)| + 2\,.
\]
Moreover, we know from Proposition \ref{prop2_8}\,(3) that we must have $i(h([c]), [b]) = 0$.
But this is only possible if  $|l_1| + |k_2|=0$, that is, $l_1 = k_2 = 0$ and $h=1$. 
But now $0 = i(h([c]), [b]) = i([c], [b]) = 2$, which is a contradiction. 

\begin{figure}[ht!]
\begin{center}
\includegraphics[width=5.6cm]{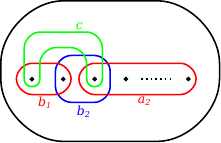}
\caption{The curves $b_1$, $b_2$, $a_2$ and $c$.}\label{fig3_4}
\end{center}
\end{figure}

Hence, $k_1$ is even, say $k_1 = 2 l_1$ with $l_1 \in \Z$. 
Suppose again that $b_1$ is a regular boundary of $e_1$ and $b_2$ a regular boundary of $e_2$ (see Figure \ref{fig3_4}). 
Recall that, by Proposition \ref{prop2_3}, $t_1^2=\sigma_{e_1}^2 = \tau_{b_1}$ and $t_2^2 = \sigma_{e_2}^2 = \tau_{b_2}$. 
In particular, $g = \tau_{b_1}^{l_1} \tau_{a_2}^{k_2}$. 
Since $s_2$ and $s_2'$ commute, $\varphi (s_2) = t_2^\varepsilon$ and $\varphi (s_2') = g t_2^\varepsilon g^{-1}$ commute, hence $t_2^2 = \tau_{b_2}$ and $g t_2^2 g^{-1} = g \tau_{b_2} g^{-1}$ commute. 
By Theorem \ref{thm2_9}, the curves $b_1$ and $b_2$ are in minimal position, hence $ i([b_1], [b_2]) = 2$. 
Similarly, $a_2$ and $b_2$ are in minimal position, hence $i ([a_2], [b_2]) = 2$. 
By Corollary \ref{corl2_11}, $i ([b_2], g ( [b_2])) = 4(|l_1| + |k_2|)$, by Proposition \ref{prop2_8}\,(1), $g t_2^2 g^{-1} = \tau_{g ([b_2])}$, and by Proposition \ref{prop2_8}\,(3), $\tau_{[b_2]}$ and $\tau_{g ([b_2])}$ commute if and only if $i ([b_2], g ([b_2]))=0$. 
We conclude that $l_1 = k_2 = 0$, hence that $g = 1$ and $\varphi (s_2') = t_2^\varepsilon$.

{\it Case 3 : $p=n-2$}.
This case is handled in the same way as Case 2.
\end{proof}

\begin{claim}\label{claim6}
We obtain a contradiction.
\end{claim}

\begin{proof}
By Claims \ref{claim1} to \ref{claim5}, we can suppose that there exists $\varepsilon \in \{ \pm 1\}$ such that
\begin{gather*}
\varphi (s_i) = t_i^\varepsilon \text{ for each } i \in \{1, \dots, p-1, p+1, \dots, n-1\}\,,\\
\varphi (s_p) = \varphi (s_p') = t_p^\varepsilon\,.
\end{gather*}
We suppose that $\varepsilon=1$. 
The case where $\varepsilon=-1$ is proved in the same way.

By definition, $s_{p-1}$ and $s_p^{-1} s_{p+1} s_p' s_{p+1}^{-1} s_p$ commute 
in $Q_{n,p}$, hence $\varphi(s_{p-1}) = t_{p-1}$ and $\varphi (s_p^{-1} s_{p+1}
\allowbreak
s_p' s_{p+1}^{-1} s_p) =  t_p^{-1} t_{p+1} t_p t_{p+1}^{-1} t_p$ commute, hence $t_{p-1}^2$ and  $t_p^{-1} t_{p+1} t_p^2 t_{p+1}^{-1} t_p$ commute. 
Denote by $b_{p-1}$ a regular boundary of $e_{p-1}$ and by $b_p$ a regular boundary of $e_p$. 
By Proposition \ref{prop2_3}, $t_{p-1}^2 = \tau_{b_{p-1}}$ and $t_p^2 = \tau_{b_p}$. 
Let $c$ be a simple closed curve such that $[c] = (\sigma_{e_p}^{-1} \sigma_{e_{p+1}}) (b_p)$, as shown in Figure \ref{fig3_5}. 
By Proposition \ref{prop2_8}\,(1), $t_p^{-1} t_{p+1} t_p^2 t_{p+1}^{-1} t_p = \tau_c$. 
By Theorem \ref{thm2_9}, $b_{p-1}$ and $c$ are in minimal position, hence $i([b_{p-1}], [c])=4$. 
By Proposition \ref{prop2_8}\,(3),  this implies that $t_{p-1}^2 = \tau_{b_{p-1}}$ and $t_p^{-1} t_{p+1} t_p^2 t_{p+1}^{-1} t_p = \tau_c$ do not commute, and so we have a contradiction. 
\end{proof}

Clearly, Claim \ref{claim6} completes the proof of Theorem \ref{thm1_1}.
\end{proof}

\begin{figure}[ht!]
\begin{center}
\includegraphics[width=8.8cm]{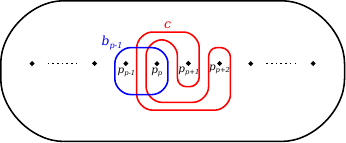}
\caption{The curves $b_{p-1}$ and $c$.}\label{fig3_5}
\end{center}
\end{figure}


\end{document}